\documentclass{amsart}
\usepackage{amssymb,amsfonts,latexsym,amsmath,amscd}

\usepackage[all]{xy}
\usepackage{verbatim}

\newtheorem{theorem}{Theorem}[section]
\newtheorem{lemma}[theorem]{Lemma}
\newtheorem{proposition}[theorem]{Proposition}
\newtheorem{corollary}[theorem]{Corollary}
\theoremstyle{definition}

\newtheorem{example}[theorem]{Example}

\newtheorem{remark}[theorem]{Remark}

\newcommand{\End}{\text{End}}
\newcommand{\Supervect}{\text{Supervect}}

\newcommand{\Ext}{\text{\rm Ext}^1}

\newcommand{\Hom}{\text{Hom}}

\newcommand{\Rep}{\text{Rep}}

\newcommand{\ot}{\otimes}

\newcommand{\ben}{\begin{enumerate}}
\newcommand{\een}{\end{enumerate}}

\newcommand{\C}{{\mathcal C}}

\newcommand{\unit}{{\bf 1}}

\begin{document}

\title[Semisimplicity in symmetric rigid tensor categories]
{Semisimplicity in symmetric rigid tensor categories}


\author{Shlomo Gelaki}
\address{Department of Mathematics, Technion-Israel Institute of
Technology, Haifa 32000, Israel} \email{gelaki@math.technion.ac.il}

\date{\today}


\dedicatory{Dedicated to Susan Montgomery in honor of her 65th
birthday}

\begin{abstract}
Let $\lambda$ be a partition of a positive integer $n$. Let $\C$ be
a symmetric rigid tensor category over a field $k$ of characteristic
$0$ or $char(k)>n$, and let $V$ be an object of $\C$. In our main
result (Theorem \ref{main}) we introduce a finite set of integers
$F(\lambda)$ and prove that if the Schur functor
$\mathbb{S}_{\lambda}V$ of $V$ is semisimple and the dimension of
$V$ is not in $F(\lambda)$, then $V$ is semisimple. Moreover, we
prove that for each $d\in F(\lambda)$ there exist a symmetric rigid
tensor category $\C$ over $k$ and a non-semisimple object $V\in \C$
of dimension $d$ such that $\mathbb{S}_{\lambda}V$ is semisimple
(which shows that our result is the best possible). In particular,
Theorem \ref{main} extends two theorems of Serre for $\C=\Rep(G)$,
$G$ is a group, and $\mathbb{S}_{\lambda}V$ is $\wedge^n V$ or
$Sym^n V$, and proves a conjecture of Serre (\cite{s1}).
\end{abstract}

\maketitle

\section{Introduction}

Let $G$ be any group, let $k$ be a field and let $\Rep(G)$ be the
category of finite dimensional representations of $G$ over $k$. A
classical result of Chevalley states that in characteristic $0$, the
tensor product $V\otimes W$ of any two semisimple objects $V,W\in
\Rep(G)$ is also semisimple \cite{c}. Later on, Serre proved that
this is also the case in positive characteristic $p$, provided that
$\dim V +\dim W<p-2$ \cite{s2}.

In \cite{s1}, Serre considered the ``converse theorems", and proved
that $V\in \Rep(G)$ is semisimple in each one of the following
situations: there exists $W\in \Rep(G)$ such that $\dim W\ne 0$ in
$k$ and $V\otimes W$ is semisimple (Theorem 2.4, {\em loc.cit}),
$V^{\otimes n}$ is semisimple for some $n\ge 1$ (Theorem 3.4, {\em
loc.cit}), $\wedge^n V$ is semisimple for some $n\ge 1$ and $\dim
V\ne 2,\dots,n$ in $k$ (Theorem 5.2.5, {\em loc.cit}), or $Sym^n V$
is semisimple for some $n\ge 1$ and $\dim V\ne -n,\dots,-2$ (Theorem
5.3.1, {\em loc.cit}).

Furthermore, Serre comments that it is easy to check that all the
above mentioned results from \cite{s1} extend to categories of
linear representations of Lie algebras and restricted Lie algebras
(when $p>0$) ({\em loc.cit}, p. 510). Moreover, Serre explains how
to extend his results Theorem 2.4 and Theorem 3.4 ({\em loc.cit}) to
any symmetric rigid tensor category over $k$, and says on p.511
({\em loc.cit}): ``I have not managed to rewrite the proofs in
tensor category style. Still, I feel that Theorem 5.2.5 on $\wedge^n
V$ and Theorem 5.3.1 on $Sym^n V$ should remain true whenever $n!\ne
0$ in $k$, i.e., $p=0$ or $p>n$." This paper originated in an
attempt to prove this conjecture of Serre.

A further natural generalization of Serre's results would be to
consider any Schur functor $\mathbb{S}_{\lambda}$, and not only
$\wedge^n$ and $Sym^n$. Namely, to look for an extension of Theorem
5.2.5 and Theorem 5.3.1 in \cite{s1}, where $\C$ is any symmetric
rigid tensor category over $k$, and $V\in \C$ is an object for which
$\mathbb{S}_{\lambda}V$ is semisimple for some partition $\lambda$
of $n$. This is precisely the main purpose of this paper.

The paper is organized as follows. In Section 2 we note that in fact
Theorem 2.4 from \cite{s1} holds in a much more general situation
than the symmetric one. More precisely, let $\mathcal{C}$ be {\em
any} rigid tensor category, and suppose that $W\in\C$ is isomorphic
to its double dual $W^{**}$ via an isomorphism $i$. This allows to
define a scalar $\dim_i(W)$ in $k$, and we show that if
$\dim_i(W)\ne 0$ and $V\otimes W$ is semisimple, then $V$ is
semisimple (see Theorem \ref{s3}). Examples, other than
$\C=\Rep(G)$, are given by braided rigid tensor categories $\C$ and
by representation categories $\C$ of Hopf algebras whose squared
antipode is inner.

In Section 3 we note that Theorem 3.3 and Corollary 3.4 from
\cite{s1} hold in a much more general situation than the symmetric
one, as well. More precisely, let $\mathcal{C}$ be {\em any} rigid
tensor category satisfying the {\em commutativity condition}, and
let $V\in \C$. We show that if $V^{\otimes n}\otimes V^{*\otimes m}$
is semisimple for some $m,n\ge 0$, not both equal to $0$, then $V$
is semisimple. In particular, if $V^{\otimes n}$ is semisimple for
some $n\ge 1$ then $V$ is semisimple (see Theorem \ref{tensor}).
Examples, other than $\C=\Rep(G)$, are given by braided rigid tensor
categories $\C$.

In Section 4 we state the main result of the paper (Theorem
\ref{main}), and prove various results in preparation for its proof.
Our main result extends Theorem 5.2.5 on $\wedge^n$ and Theorem
5.3.1 on $Sym^n$ in the group case $\C=\Rep(G)$ \cite{s1}, to any
symmetric rigid tensor category $\C$ over $k$ and any Schur functor
$\mathbb{S}_{\lambda}$ (so, in particular, it provides a proof to
the conjecture of Serre (\cite{s1}, p.511)). More precisely, let
$\lambda$ be a partition of a positive integer $n$, and assume that
$char(k)=0$ or $char(k)>n$. Let $\mathbb{S}_{\lambda}$ be the
associated Schur functor (see \cite{d2}) and let $V$ be an object of
$\C$. In Theorem \ref{main} we introduce a finite set of integers
$F(\lambda)$ and prove that if the dimension of $V$ is not equal in
$k$ to an element of $F(\lambda)$ and $\mathbb{S}_{\lambda}V$ is
semisimple, then $V$ is semisimple. Moreover, we prove that for each
$d\in F(\lambda)$ there exist a symmetric rigid tensor category $\C$
over $k$ and a non-semisimple object $V\in \C$ of dimension $d$ such
that $\mathbb{S}_{\lambda}V$ is semisimple (which shows that our
result is the best possible).

Section 5 is devoted to the proof of Theorem \ref{main}.

All tensor categories will be assumed to be rigid, $k$-linear
Abelian, with finite dimensional $\Hom$ spaces, such that every
object has a finite length, and $\End(\unit)=k$.

{\bf Acknowledgments.} The author is grateful to Pavel Etingof for
his encouragement, for his interest in the problem, and for
stimulating and helpful discussions.

The author was supported by The Israel Science Foundation (grant No.
317/09).

\section{From $V\otimes W$ to $V$ in rigid
tensor categories}

Let $\mathcal{C}$ be a rigid tensor category. For an object $V\in
\mathcal{C}$ we let
$$coev_V: \unit \to V\otimes V^*\,\,\text{and}\,\, ev_V:
V^*\otimes V\to \unit$$ denote the coevaluation and evaluation maps
associated to $V$, respectively. Recall that $$(id_V\otimes
ev_V)\circ (coev_V\ot id_V)=id_V.$$

The following two propositions were proved by Serre for
$\mathcal{C}:=\Rep(G)$, $G$ any group \cite{s1}. However, it is
straightforward to verify that the same proofs work in any rigid
tensor category $\mathcal{C}$.

\begin{proposition}\label{s1} (\cite{s1}, Proposition 2.1)
Let $V,W\in \mathcal{C}$, and let $V'$ be a sub-object of $V$.
Assume that $coev_{W}: \unit \to W\otimes W^*$ and $V'\otimes W\to
V\otimes W$ are split injections. Then $V'\to V$ is a split
injection.
\end{proposition}

\begin{proposition}\label{s2} (\cite{s1}, Proposition 2.3)
Assume $coev_{W}: \unit \to W\otimes W^*$ is a split injection and
that $V\otimes W$ is semisimple. Then $V$ is semisimple.
\end{proposition}

One instance in which $coev_{W}: \unit \to W\otimes W^*$ is a split
injection is the following. Assume that $W\in \mathcal{C}$ is
isomorphic to its double dual $W^{**}$, and fix an isomorphism
$i:W\to W^{**}$. This allows us to define the (quantum) dimension
$\dim_i(W)$ of $W$ (relative to $i$) as the composition
$$\dim_i(W):= ev_{W^*}\circ (i\otimes id_{W^*})\circ coev_{W}.$$ Note that
$\dim_i(W)\in \End(\unit)=k$. Now, clearly if $\dim_i(W)\ne 0$ in
$k$, then $coev_{W}$ is a split injection. (See Remark 2.2 in
\cite{s1}.)

As a consequence of Proposition \ref{s1} and Proposition \ref{s2} we
have the following theorem, which generalizes Theorem 2.4 in
\cite{s1}.

\begin{theorem}\label{s3}
Assume that $W\in\C$ is isomorphic to its double dual $W^{**}$ and
let $i:W\to W^{**}$ be an isomorphism. If $V\otimes W$ is semisimple
and $\dim_i(W)\ne 0$ in $k$ then $V$ is semisimple.
\end{theorem}

\begin{remark}
1) It is known that if $\C$ is braided, any object $W$ is isomorphic
to its double dual $W^{**}$. So in particular, if $H$ is a
quasitriangular (quasi)Hopf algebra over $k$ and $V,W\in \Rep(H)$
such that $V\otimes W$ is semisimple and $\dim W\ne 0$ in $k$, then
$V$ is semisimple. The converse is not true.

2) If $H$ is a Hopf algebra whose squared antipode $S^2$ is inner
(e.g., $S^2=id$) then any $W\in \Rep(H)$ is isomorphic to $W^{**}$.
Therefore Theorem \ref{s3} holds for $\Rep(H)$.

3) When $\C$ is symmetric, Serre already pointed out that Theorem
2.4 in \cite{s1} holds for $\C$, with the same proof (see p.
510--511 in \cite{s1}).
\end{remark}

\section{From $V^{\otimes n}\otimes V^{*\otimes m}$ to $V$ in rigid
tensor categories}

The following theorem was proved by Serre for
$\mathcal{C}:=\Rep(G)$, $G$ any group \cite{s1}. Serre also explains
that the same proof works in any symmetric rigid tensor category
$\C$. In fact, the symmetry is used only to guarantee that for any
$V\in \mathcal{C}$ the morphism
$$id_V\ot coev_V:V\to V\otimes V\otimes V^*$$ is a split injection.
We just note that in fact this is the case in any rigid tensor
category $\C$ satisfying the following {\em commutativity
condition}: there exists a functorial isomorphism $c: \otimes\to
\otimes^{op}$ such that $c_{V\ot\unit}=c_{\unit\ot V}=id_V$ for any
$V\in \C$ (e.g., $\C$ is braided, not necessarily symmetric).
Indeed, let $\C$ be a rigid tensor category satisfying the
coboundary condition. Then, using the naturality of $c$, one has
\begin{equation}\label{yan}
(id_V\otimes ev_V)\circ c_{V,V\ot V^*}\circ (id_V\ot
coev_V)=(id_V\otimes ev_V)\circ (coev_V\ot id_V)=id_V.
\end{equation}

Therefore we have the following result, which generalizes Theorem
3.3 and Corollary 3.4 in \cite{s1}.

\begin{theorem}\label{tensor}
Let $\C$ be a rigid tensor category satisfying the commutativity
condition, and let $V\in \C$. If $V^{\otimes n}\otimes V^{*\otimes
m}$ is semisimple for some $m,n\ge 0$, not both equal to $0$, then
$V$ is semisimple. In particular, if $V^{\otimes n}$ is semisimple
for some $n\ge 1$ then $V$ is semisimple. \qed
\end{theorem}

\section{from $\mathbb{S}_{\lambda}V$
to $V$ in symmetric rigid tensor categories}

In this section we assume that $\C$ is a {\em symmetric} rigid
tensor category over a field $k$, with a commutativity constraint
$c$ (see e.g., \cite{d1}, \cite{d2}).

\subsection{Schur functors in $\C$.}
Recall that given an object $X\in \C$ and a nonnegative integer $m$,
the symmetric group $S_m$ acts on $X^{\ot m}$ via the symmetry $c$.
Let $\beta$ be a partition of $m$, and assume that $char(k)>m$ if
$char(k)\ne 0$. Let $V_{\beta}$ be the corresponding irreducible
representation of $S_m$ and let $c_{\beta}\in k[S_m]$ be a Young
symmetrizer associated with $V_{\beta}$. Then $c_{\beta}$ gives rise
to a functor
$$c_{\beta}:\C\to \C,\,\,X\mapsto c_{\beta}(X^{\otimes
m}).$$ Recall that the isomorphism type of the functor $c_{\beta}$
does not depend on the choice of $c_{\beta}$. We shall call
$\mathbb{S}_{\beta}X:=c_{\beta}(X^{\otimes m})\subseteq X^{\otimes
m}$ the {\em Schur functor} of $X$ associated with $\beta$.

Schur functors in symmetric rigid tensor categories were introduced
(more conceptually) and studied by Deligne in \cite{d2}. Among many
other things, it is proved there that for any object $X\in \C$,
$(\mathbb{S}_{\beta}X)^*$ is canonically isomorphic to
$\mathbb{S}_{\beta}X^*$, a fact we shall use often in the sequel.

\begin{example}
Note in particular that $\mathbb{S}_{(0)}X=\unit$,
$\mathbb{S}_{(1)}X=X$, $\mathbb{S}_{(m)}X= Sym^mX$ and
$\mathbb{S}_{(1^m)}X= \wedge^mX$.
\end{example}

\subsection{The main result.}
Our goal is to generalize Theorem 5.2.5 and Theorem 5.3.1 from
\cite{s1} by replacing representations categories $\Rep(G)$ of
groups by any symmetric rigid tensor category $\C$, and by replacing
the Schur functors $\wedge^n$, $Sym^n$ by any Schur functor. More
precisely, let $\lambda$ be a partition of a positive integer $n$
and let $V\in \C$. Our goal is to find out when the semisimplicity
of $\mathbb{S}_{\lambda}V$ implies the semisimplicity of $V$, in
terms of the dimension of $V$ only.

Fix a partition $\lambda$ of a positive integer $n$, with
$p:=p(\lambda)$ rows and $q:=q(\lambda)$ columns, and let $(i,j)$
number the row and column of boxes for the Young diagram of
$\lambda$. Let us introduce some notation.

\begin{itemize}
\item Let $R(\lambda)$ denote the integral interval $\{-q,\dots,p\}$,
and let $T(\lambda)\subseteq R(\lambda)$ include $0$ if $\lambda$ is
a hook (i.e., $(2,2)\notin \lambda$), $1$ if $(3,2)\notin \lambda$,
$-1$ if $(2,3)\notin \lambda$, and $-q,p$ if $\lambda$ is not a
rectangle. Set $F(\lambda):=R(\lambda)\setminus T(\lambda)$.

\item Let $G(\lambda)$ denote the set of all values $d$ in $k$ for
which there exists a symmetric rigid tensor category $\C$ over $k$
with a non-semisimple object $V$ of dimension $d$ such that
$\mathbb{S}_{\lambda}V$ is semisimple.
\end{itemize}

\begin{remark}\label{trick}
1) We have that $F(\lambda)=-F(\lambda^*)$, where $\lambda^*$ is the
conjugate of $\lambda$.

2) We have that $G(\lambda)=-G(\lambda^*)$. Indeed, if $(\C,V)$ is a
counterexample for $(\lambda,d)$ (i.e., $\C$ is a symmetric rigid
tensor category over $k$ with a non-semisimple object $V$ of
dimension $d$ such that $\mathbb{S}_{\lambda}V$ is semisimple) then
$(\C\boxtimes \Supervect,V\otimes \unit^{-1})$ is a counterexample
for $(\lambda^*,-d)$, where $\Supervect$ is the category of finite
dimensional super vector spaces over $k$ and
$\unit^{-1}\in\Supervect$ is the odd $1-$dimensional space.
\end{remark}

We can now state our main result concisely.

\begin{theorem}\label{main}
Let $n$ be a positive integer, $n<char(k)$ in case $char(k)\ne 0$,
and let $\lambda$ be a partition of $n$. Then the sets $F(\lambda)$
and $G(\lambda)$ coincide (where we view the relevant integers as
elements of $k$ in an obvious way).
\end{theorem}

\begin{example}
Let $\C$ be a symmetric rigid tensor category over $k$, and let
$V\in \C$.

1) Theorem \ref{main} implies for $\lambda=(1^n)$ (respectively,
$\lambda=(n)$), that if $\mathbb{S}_{\lambda}V$ is semisimple and
the dimension of $V$ is not equal in $k$ to an integer in the range
$2,\dots,n$ (respectively, $-n,\dots,-2$), then $V$ is semisimple.
For $\C=\Rep(G)$, $G$ is any group, this is Theorem 5.2.5 from
\cite{s1} (respectively, Theorem 5.3.1 from \cite{s1}).

2) Theorem \ref{main} implies that if $\mathbb{S}_{(2,1)}V$ is
semisimple then so is $V$.
\end{example}

The proof of Theorem \ref{main} is given in Section 5. The rest of
this section is devoted to preparations for the proof.

\subsection{Traces in $\C$.}
For an object $X\in \mathcal{C}$, let $$\widetilde{ev}_X:=ev_X\circ
c_{X,X^*}: X\ot X^*\to \unit.$$ Recall that the dimension $\dim X\in
k$ of $X$ is defined by $$\dim X:=\widetilde{ev}_X\circ
coev_X:\unit\to \unit.$$

In \cite{jsv} it is explained that the family of functions
$$Tr_{A,B}^U:\Hom(A\ot U,B\ot U)\to \Hom(A,B),\,\,\, A,B,U\in \C,$$
defined by
\begin{equation}
\label{trace} Tr_{A,B}^U(f):A \xrightarrow{id_A\ot coev_U} A\ot U\ot
U^* \xrightarrow{f\ot id_{U^*}} B\ot U\ot U^*\xrightarrow{id_B\ot
\widetilde{ev}_U} B,
\end{equation}
is natural in $U$, $A$ and $B$, and satisfies the following property
(among other properties)
\begin{equation}\label{van}
Tr_{A,B}^{U\ot W}(f)=Tr_{A,B}^{U}(Tr_{A\ot U,B\ot U}^{W}(f)).
\end{equation}

Clearly, $Tr_{\unit,\unit}^U(id_U)=\dim U$.

We have the following two easy lemmas.
\begin{lemma}\label{trace1}
Let $f:A\ot U\to B\ot W$ and $g:W\to U$ be morphisms. Then
$Tr_{A,B}^U((id_B\ot g)f)=Tr_{A,B}^W(f(id_A\ot g))$.
\end{lemma}

\begin{proof}
Follows from the naturality of $Tr$ in $U$.
\end{proof}

\begin{lemma}\label{trace2}
Let $f:A\ot U\to B\ot U$ and $g:W\to W$ be morphisms. Then
$Tr_{A,B}^U(f)\ot Tr_{\unit,\unit}^W(g)=Tr_{A,B}^{U\ot W}(f\ot g)$.
\end{lemma}

\begin{proof}
Follows easily from the definition of $Tr$, and the facts that
$(U\ot W)^*=W^*\ot U^*$ with $$coev_{U\ot W}=(id_U\ot c_{U^*,W\ot
W^*})\circ(coev_U\ot coev_W)$$ and $$\widetilde{ev}_{U\ot
W}=(\widetilde{ev}_U\ot \widetilde{ev}_W)\circ(id_U\ot c_{W\ot
W^*,U^*})$$ (see e.g., \cite{bk}).
\end{proof}

\subsection{Traces of permutations.}
Fix a nonnegative integer $m$, and an object $X\in \C$. In the
sequel we shall identify the symmetric group $S_{m-1}$ with the
stabilizer of $1$ in $S_{m}$.

\begin{lemma}\label{l1}
For any $\sigma\in S_{m}$ and $\tau\in S_{m-1}$, $Tr_{X,X}^{X^{\ot
m-1 }}(\sigma)=Tr_{X,X}^{X^{\ot m-1 }}(\tau\sigma\tau^{-1})$.
\end{lemma}

\begin{proof}
Follows easily from Lemma \ref{trace1}.
\end{proof}

\begin{lemma}\label{l3}
We have that $Tr_{X,X}^{X^{\ot m-1 }}((1\cdots m))=id_X$.
\end{lemma}

\begin{proof}
For any $i$ let us denote the cycle $(1\cdots i)$ by $\sigma_i$. We
are going to prove the lemma by induction on $m$ using the relation
$\sigma_{m}=(12)\sigma_{m-1}$. We compute
\begin{eqnarray*}
\lefteqn { Tr_{X,X}^{X^{\ot m-1}}(\sigma_{m})
}\\
& = & Tr_{X,X}^{X}(Tr_{X\ot X,X\ot X}^{X^{\ot m -2}}(\sigma_{m}))\\
& = & Tr_{X,X}^{X}(Tr_{X\ot X,X\ot X}^{X^{\ot m -2}}(((12)\ot id)
\circ(id\ot \sigma_{m-1})))\\
& = & Tr_{X,X}^{X}((12)\circ Tr_{X\ot X,X\ot X}^{X^{\ot m -2}}(id\ot
\sigma_{m-1}))\\
& = & Tr_{X,X}^{X}((12)\circ (id_X\ot Tr_{X,X}^{X^{\ot m -2}}
(\sigma_{m-1})))\\
& = & Tr_{X,X}^{X}((12)\circ (id_X\ot id_X))\\
& = & id_X,
\end{eqnarray*}
where in the first equality we used (\ref{van}), in the third
equality we used the naturality of $Tr$ in $X\ot X$, in the fifth
equality we used the induction assumption, and in the last equality
we used (\ref{yan}).
\end{proof}

\begin{lemma}\label{l2}
Let $\sigma_1\sigma_2\cdots\sigma_{N}\in S_{m}$ be a product of
disjoint cycles, where reading from left to right the numbers
$1,\dots,m$ appear in an increasing order. Then $Tr_{X,X}^{X^{\ot
m-1}}(\sigma_1\sigma_2\cdots\sigma_{N})=d^{N-1}id_X$, where $d$ is
the dimension of $X$.
\end{lemma}

\begin{proof}
Lemma \ref{l3} is the case $N=1$. Now use Lemma \ref{trace2} to
proceed by induction on $N$.
\end{proof}

\begin{proposition}\label{c2}
Let $\sigma\in S_{m}$, and let $N(\sigma)$ denote the number of
disjoint cycles in $\sigma$. Then $Tr_{X,X}^{X^{\ot
m-1}}(\sigma)=d^{N(\sigma)-1}id_X$, where $d:=\dim X$.
\end{proposition}

\begin{proof}
It is clear that for any $\sigma\in S_{m}$ there exists $\tau\in
S_{m-1}$ such that $\tau\sigma\tau^{-1}$ decomposes into a product
of disjoint cycles $\sigma_1\sigma_2\cdots\sigma_{N(\sigma)}$, where
reading from left to right the numbers $1,\dots,m$ are in an
increasing order. Now, by Lemma \ref{l1}, $Tr_{X,X}^{X^{\ot m-1
}}(\sigma)=Tr_{X,X}^{X^{\ot m-1 }}(\tau\sigma\tau^{-1})$, and hence
the result follows from Lemma \ref{l2}.
\end{proof}

\subsection{The morphism $\theta_{X,m,\alpha,\beta}$.}
Given a partition $\alpha$ of a nonnegative integer $m-1$, let
$\alpha+1$ denote the set of partitions of $m$ whose Young diagram
is obtained by adding a single box to the Young diagram of $\alpha$.

Fix an object $X\in \C$ of dimension $d:=\dim X$, and partitions
$\alpha$ of $m-1$ and $\beta\in \alpha+1$. We define the morphism
$$\theta_{\alpha,\beta}=
\theta_{X,m,\alpha,\beta}:X\to \mathbb{S}_{\beta}X\ot
\mathbb{S}_{\alpha}X^*$$ as the following composition:
\begin{equation}\label{theta}
\theta_{\alpha,\beta}:X  \xrightarrow{id_X\ot
coev_{\mathbb{S}_{\alpha}X}} X\ot \mathbb{S}_{\alpha}X\ot
\mathbb{S}_{\alpha}X^* \xrightarrow{c_{\beta}\ot c_{\alpha}}
\mathbb{S}_{\beta}X\ot \mathbb{S}_{\alpha}X^*.
\end{equation}

Consider the morphism
\begin{equation*}\label{spinj}
P_{\alpha,\beta}=P_{X,m,\alpha,\beta}:X\to X,
\end{equation*}
given as the composition
\begin{equation}\label{spinj}
P_{\alpha,\beta}:X
\xrightarrow{\theta_{\alpha,\beta}}\mathbb{S}_{\beta}X\ot
\mathbb{S}_{\alpha}X^* \hookrightarrow X\ot X^{\ot m-1 }\ot X^{*\ot
m-1 } \xrightarrow{id_X\ot \widetilde{ev}_{X^{\ot m-1 }}} X.
\end{equation}
In what follows we shall see that the morphism $P_{\alpha,\beta}$ is
a scalar multiple of the identity morphism $id_X$ by some polynomial
$p_{\alpha,\beta}(d)$.

If we identify $S_{m-1}$ with the stabilizer of $1$ in $S_{m}$, then
clearly
$$P_{\alpha,\beta}=Tr_{X,X}^{X^{\ot m-1 }}((id_X\ot c_{\alpha})
\circ c_{\beta}\circ (id_X\ot c_{\alpha})):X\to X,$$ and hence, by
Lemma \ref{trace1},
$$P_{\alpha,\beta}=Tr_{X,X}^{X^{\ot m-1}}((id_X\ot c_{\alpha})\circ
c_{\beta}):X\to X.$$

As an immediate consequence of Proposition \ref{c2}, we get the
following.
\begin{corollary}
Write $(id_X\ot c_{\alpha})\circ c_{\beta}\in k[S_{m}]$ as a
$k-$linear combination of group elements: $(id_X\ot c_{\alpha})\circ
c_{\beta}=\sum_{\sigma\in S_{m}}f_{\alpha,\beta}(\sigma)\sigma$, and
set $$p_{\alpha,\beta}(d):=\sum_{\sigma\in
S_{m}}f_{\alpha,\beta}(\sigma)d^{N(\sigma)-1}.$$ Then
$P_{\alpha,\beta}=p_{\alpha,\beta}(d)id_X$. In particular, if
$p_{\alpha,\beta}(d)\ne 0$ in $k$ then $\theta_{\alpha,\beta}$ is a
split injection. \qed
\end{corollary}

Let $\chi_{\beta}$ be the character of $V_{\beta}$, and let
$$e_{\beta}:=\frac{\dim V_{\beta}}{m!}\sum_{\sigma\in
S_m}\chi_{\beta}(\sigma) \sigma$$ be the primitive central
idempotent in $k[S_m]$ associated with $V_{\beta}$. Recall that
$e_{\beta}$ is equal to the sum of all the ($\dim V_{\beta}$) Young
symmetrizers $c_{\beta}$ associated with $V_{\beta}$.

In the following theorem we compute the polynomial
$p_{\alpha,\beta}(d)$ explicitly, in terms of $\chi_{\beta}$.

\begin{theorem}\label{c1}
We have that
$$Tr_{X,X}^{X^{\ot m-1 }}\left((id_X\ot
e_{\alpha})\circ e_{\beta}\right)=\left(\frac{\dim
V_{\alpha}}{m!}\sum_{\sigma\in S_{m}}\chi_{\beta}(\sigma)
d^{N(\sigma)-1}\right)id_X,$$ and hence
$$p_{\alpha,\beta}(d)=\frac{1}{m!\dim
V_{\beta}}\sum_{\sigma\in S_{m}}\chi_{\beta}(\sigma)
d^{N(\sigma)-1}.$$
\end{theorem}

\begin{proof}
Clearly, $$(id_X\ot e_{\alpha})\circ e_{\beta}=\frac{\dim
V_{\alpha}\dim V_{\beta}}{(m-1)!m!} \sum_{\sigma\in
S_{m}}\left(\sum_{\tau\in S_{m-1}}\chi_{\alpha}(\tau)
\chi_{\beta}(\tau^{-1}\sigma)\right)\sigma.$$ Therefore, by
Proposition \ref{c2},
\begin{eqnarray*}
\lefteqn{Tr_{X,X}^{X^{\ot m-1}}((id_X\ot e_{\alpha})\circ e_{\beta})}\\
= & & \left(\frac{\dim V_{\alpha}\dim V_{\beta}}{(m-1)!m!}
\sum_{\sigma\in S_{m}}\left(\sum_{\tau\in
S_{m-1}}\chi_{\alpha}(\tau)
\chi_{\beta}(\tau^{-1}\sigma)\right)d^{N(\sigma)-1}\right)id_X\\
= & & \left(\frac{\dim V_{\alpha}\dim V_{\beta}}{(m-1)!m!}
\left(\sum_{\tau\in S_{m-1}}\chi_{\alpha}(\tau) \chi_{\beta}\left
(\tau^{-1}\left(\sum_{\sigma\in S_{m}}
d^{N(\sigma)-1}\sigma\right)\right)\right)\right)id_X.
\end{eqnarray*}

Set $z(d):=\sum_{\sigma\in S_{m}}d^{N(\sigma)-1}\sigma$. Clearly,
$z(d)$ is a central element in $k[S_{m}]$, hence it acts by the
scalar $\chi_{\beta}(z(d))/\dim V_{\beta}$ on $V_{\beta}$. In
particular, for any $\tau\in S_{m}$, $\chi_{\beta}(\tau^{-1}z(d))=
\chi_{\beta}(\tau^{-1})\chi_{\beta}(z(d))/\dim V_{\beta}$. We
therefore have
\begin{eqnarray*}
\lefteqn{Tr_{X,X}^{X^{\ot m-1}}((id_X\ot e_{\alpha})\circ e_{\beta})}\\
= & & \left(\frac{\dim V_{\alpha}}{(m-1)!m!} \left(\sum_{\tau\in
S_{m-1}}\chi_{\alpha}(\tau)
\chi_{\beta}(\tau^{-1})\right)\chi_{\beta}(z(d))\right)id_X.
\end{eqnarray*}

Finally, recall that the multiplicity
$[Res^{S_{m}}_{S_{m-1}}\chi_{\beta}:\chi_{\alpha}]$ of $V_{\alpha}$
in the restriction of $V_{\beta}$ from $S_{m}$ to $S_{m-1}$ is equal
to $1$ (see e.g. \cite{fh}), i.e.,
$$\frac{1}{(m-1)!}\sum_{\tau\in
S_{m-1}}\chi_{\alpha}(\tau) \chi_{\beta}(\tau^{-1})=
[Res^{S_{m}}_{S_{m-1}}\chi_{\beta}:\chi_{\alpha}]=1.$$ We thus
conclude that
$$Tr_{X,X}^{X^{\ot m-1 }}((id_X\ot e_{\alpha})\circ
e_{\beta}) = \left(\frac{\dim V_{\alpha}}{m!}\sum_{\sigma\in
S_{m}}\chi_{\beta}(\sigma) d^{N(\sigma)-1}\right)id_X,$$ as claimed.
\end{proof}

In fact, the polynomial $p_{\alpha,\beta}(d)$ is closely related to
a well known polynomial associated with the partition $\beta$.
Namely, let ${\rm cp}_{\beta}(d):=\prod_{(i,j)\in \beta}(d+j-i)$ be
the content polynomial of $\beta$, and recall that the polynomial
(in $d$) $\frac{1}{\dim V_{\beta}}\sum_{\sigma\in
S_{m}}\chi_{\beta}(\sigma) d^{N(\sigma)}$ equals ${\rm
cp}_{\beta}(d)$ (see e.g. \cite{md}). Hence, by Theorem \ref{c1},
\begin{equation}\label{new}
p_{\alpha,\beta}(d)d=\frac{1}{m!}{\rm cp}_{\beta}(d).
\end{equation}

\begin{corollary}\label{cor1}
Let $\alpha$, $\beta$, $X$ and $d$ be as above, and let $p(\beta)$,
$q(\beta)$ be the number of rows and columns in the diagram of
$\beta$, respectively.

1) If $d\ne 1-q(\beta),\dots,p(\beta)-1$ in $k$ then the morphism
$\theta_{\alpha,\beta}$ is a split injection.

2) Suppose $\beta$ is a hook. If $d\ne
1-q(\beta),\dots,-1,1,\dots,p(\beta)-1$ in $k$ then the morphism
$\theta_{\alpha,\beta}$ is a split injection.
\end{corollary}

\begin{proof}
1) Since $d\ne 0$ in $k$, the result follows from (\ref{new}) and
Theorem \ref{c1}.

2) By Theorem \ref{c1}, $p_{\alpha,\beta}(0)=\frac{1}{m!\dim
V_{\beta}}\sum_{\sigma}\chi_{\beta}(\sigma)$, where the sum is taken
over all the $m-$cycles $\sigma$ in $S_{m}$. But it is well known
(see e.g. \cite{md}) that $\chi_{\beta}$ vanishes on a $m-$cycle
when $\beta$ is not a hook, and that
$\chi_{(m-s,1^s)}(\sigma)=(-1)^s$ for any $0\le s\le m$ and
$m-$cycle $\sigma$. Therefore $\theta_{\alpha,\beta}$ is a split
injection when $d=0$ as well.

We are done.
\end{proof}

\begin{example}
For the partition $\alpha=(1^{m-1})$,
$\mathbb{S}_{\alpha}X=\wedge^{m-1} X$ is the $(m-1)th$ exterior
power of $X$. Hence, by Corollary \ref{cor1}, if
${{d-1}\choose{m-1}}\ne 0$ in $k$, then the corresponding morphism
$\theta_{(1^{m-1}),(1^{m})}$ is a split injection. This is a
generalization of Lemma 5.1.12 in \cite{s1} in the group case.
\end{example}

\subsection{Extensions in $\C$.}
Let $U,V,W\in \C$ and let $f\in \Hom(V,W)$, $g\in \Hom(W,U)$. We
shall denote by $f_*$ and $g^*$ the $k-$linear maps
$$f_*:\Ext(U,V)\to \Ext(U,W)\,\,\,and\,\,\, g^*:\Ext(U,V)\to
\Ext(W,V)$$ induced by $f$ and $g$, respectively. Namely, given an
extension $$E:\,\,0\to V\to X\to U\to 0,$$ the extensions
$$f_*(E):\,\,0\to W\to Y\to U\to 0\,\,\,and\,\,\,g^*(E):\,\,0\to V\to Z\to
W\to 0$$ are obtained using the pushout
\begin{equation*}
\label{equivariantX} \xymatrix{V\ar[d]^{f}\ar[rr]&& X\ar@{-->}[d] &&\\
W\ar@{-->}[rr]&& Y&&}
\end{equation*}
and the pullback
\begin{equation*}
\label{equivariantX} \xymatrix{X\ar[rr] && U\\
Z\ar@{-->}[u]\ar@{-->}[rr]&& W \ar[u]^{g} && ,}
\end{equation*}
respectively.

We shall need the following two lemmas.

\begin{lemma}\label{ext}
For any objects $A,B,X\in \C$, the $k-$linear spaces $\Ext(B,A\ot
X)$ and $\Ext(B\ot X^*,A)$ are canonically isomorphic.
\end{lemma}

\begin{proof}
One associates to an element
\begin{equation*}
0\xrightarrow{} A\ot X\xrightarrow{} W\xrightarrow{} B\xrightarrow{}
0
\end{equation*}
in $\Ext(B,A\ot X)$ an element in $\Ext(B\ot X^*,A)$ in the
following way: since the functor $-\ot X^*:\C\to \C$ is exact,
tensoring our exact sequence with $X^*$ on the right yields the
extension
\begin{equation*}\label{equivariantX}
E:\,\,\,0\xrightarrow{} A\ot X\ot X^*\xrightarrow{} W\ot
X^*\xrightarrow{} B\ot X^*\xrightarrow{} 0.
\end{equation*}
The corresponding extension
\begin{equation*}\label{(V)}
0\xrightarrow{} A\xrightarrow{} \tilde W\xrightarrow{} B\ot
X^*\xrightarrow{} 0
\end{equation*}
in $\Ext(B\ot X^*,A)$ is given by $(id_A\ot (ev_X\circ
c_{X,X^*}))_*(E)$. This assignment defines a $k-$linear map
$\Ext(B,A\ot X)\to \Ext(B\ot X^*,A)$, and it is straightforward to
verify that its inverse map is constructed similarly, using the
exact functor $-\ot X$ and the map $(id_B\ot (c_{X,X^*}\circ
coev_X))^*$.
\end{proof}

\begin{lemma}\label{cor}
Let $\alpha$ be a partition of a nonnegative integer $m-1$, let
$\beta\in \alpha+1$ and let $A,B,X\in \C$. Suppose that
$\theta_{\alpha,\beta}=\theta_{X,m,\alpha,\beta}$ is a split
injection. Then the $k-$linear map
$$(id_{B\ot X}\ot coev_{\mathbb{S}_{\alpha}X})_*:\Ext(A,B\ot X)\to
\Ext(A,B\ot X\ot \mathbb{S}_{\alpha}X\ot \mathbb{S}_{\alpha}X^*)$$
is injective.
\end{lemma}

\begin{proof}
Indeed, since $\theta_{\alpha,\beta}$ is a split injection, we have
that
$$(id_B\ot \theta_{\alpha,\beta})_*:\Ext(A,B\ot X)\to \Ext(A,B\ot
\mathbb{S}_{\beta}X\ot \mathbb{S}_{\alpha}X^*)$$ is injective. But,
$$(id_B\ot \theta_{\alpha,\beta})_*=
(id_B\ot c_{\beta}\ot id_{\mathbb{S}_{\alpha}X^*})_*\circ (id_{B\ot
X}\ot coev_{\mathbb{S}_{\alpha}X})_*.$$ We are done.
\end{proof}

\subsection{The filtration on $\mathbb{S}_{\lambda}V$ defined by a
sub-object of $V$.} Fix a sub-object $A$ of $V$ for the rest of the
section, and consider the short exact sequence
\begin{equation}\label{exseq}
(V):\,\,\,\,0\xrightarrow{} A\xrightarrow{} V\xrightarrow{}
B\xrightarrow{} 0;
\end{equation}
it is an element in the $k-$linear space $\Ext(B,A)$. Then $(V)$
defines a filtration on $\mathbb{S}_{\lambda}V$ in the following
way. For each $0\le i\le n$ set
$$T_{i}:=\sum_{S\subseteq \{1,...,n\},\,
|S|=i}V_{S(1)}\otimes\cdots\otimes V_{S(n)},$$ where $V_{S(j)}=V$ if
$j\notin S$ and $V_{S(j)}=A$ if $j\in S$. Clearly, the $T_{i}$
define a $S_n-$equivariant filtration $T_*$ on $V^{\otimes n}$:
$$V^{\otimes n}=T_0\supseteq T_1\supseteq\cdots\supseteq T_n\supseteq
T_{n+1}=0,$$ whose composition factors are
$$T_{i}/T_{i+1}\cong\bigoplus_{S\subseteq \{1,...,n\},\,
|S|=i}V_{S,1}\otimes\cdots\otimes V_{S,n},\,\,0\le i\le n,$$ where
$V_{S,j}=B$ if $j\notin S$ and $V_{S,j}=A$ if $j\in S$.

The filtration $T_*$ induces a filtration $F_*$ on
$\mathbb{S}_{\lambda}V$:
$$\mathbb{S}_{\lambda}V=F_0\supseteq F_1\supseteq\cdots\supseteq
F_n\supseteq F_{n+1}=0,$$ where $F_{i}:=c_{\lambda}(T_{i})$ is the
image of $T_{i}$ under the Schur functor $c_{\lambda}$. Let
\begin{equation}\label{vi}
V_{i}:=F_{i}/F_{i+1},\,\,\,0\le i\le n,
\end{equation}
be the composition factors of $F_*$, and let
\begin{equation}\label{vi2}
V_{i}^2:=F_{i-1}/F_{i+1},\,\,\,1\le i\le n.
\end{equation}
Since the filtration $T_*$ is $S_n-$equivariant, we have
\begin{equation}\label{vi1}
V_{i}\cong c_{\lambda}\left(T_{i}/T_{i+1}\right)\cong
\bigoplus_{\mu\vdash i,\,\nu\vdash
n-i}N_{\mu,\nu}^{\lambda}(\mathbb{S}_{\mu} A\ot \mathbb{S}_{\nu} B),
\end{equation}
where $N_{\mu,\nu}^{\lambda}:=[Res^{S_{n}}_{S_{i}\times
S_{n-i}}V_{\lambda}:V_{\mu}\ot V_{\nu}]$ are the
Littlewood-Richardson coefficients (see e.g., \cite{fh}).

For each integer $0\le i\le n$, let $\lambda-i$ denote the set of
all partitions of $n-i$ whose Young diagram is obtained from that of
$\lambda$ after deleting $i$ boxes (by convention, $\lambda-n$
consists of one element $(0)$). By the Littlewood-Richardson rule
(see e.g., \cite{fh}), $N_{\mu,\nu}^{\lambda}=0$ if
$\mu\notin\lambda-(n-i)$ or $\nu\notin\lambda-i$. Therefore,
\begin{equation}\label{vi11}
V_{i}\cong \bigoplus_{\mu\in \lambda-(n-i),\,\nu\in
\lambda-i}N_{\mu,\nu}^{\lambda}(\mathbb{S}_{\mu} A\ot
\mathbb{S}_{\nu} B).
\end{equation}
(However, $N_{\mu,\nu}^{\lambda}$ can still equal $0$ for some pairs
$\mu\in\lambda-(n-i)$, $\nu\in\lambda-i$, e.g., for $\lambda=(2,2)$,
$N_{(1^2),(2)}^{(2,2)}=0$.)

Observe also that for any $\mu'\in \lambda-(n-i+1)$, $\mu\in
\lambda-(n-i)$ and $\nu\in \lambda-i$, $c_{\mu}$ defines a morphism
$$c_{\mu}\ot id_{\mathbb{S}_{\nu} B}
:\mathbb{S}_{\mu'} A\ot V\ot \mathbb{S}_{\nu} B\to V_{i}^2.$$

Since $V_{i}^2$ is a subquotient of $\mathbb{S}_{\lambda}V$, the
following lemma is clear.
\begin{lemma}\label{splits}
If $\mathbb{S}_{\lambda}V$ is semisimple then the exact sequence
\begin{equation}\label{(V^2)}
(V_{i}^2):\,\,\,\, 0 \xrightarrow{} V_{i}\xrightarrow{} V_{i}^2
\xrightarrow{} V_{i-1}\xrightarrow{} 0
\end{equation}
splits for any $1\le i\le n$. \qed
\end{lemma}

\subsection{The semisimlicity of $V$.}
Let $1\le i\le n$ be an integer,
$\mu'\in \lambda-(n-i+1)$ and $\nu\in \lambda-i$. Tensoring our
exact sequence $(V)$ by $\mathbb{S}_{\mu'}A$ on the left yields the
extension
\begin{equation}\label{exseq1}
E_1:\,\,\,\, 0\xrightarrow{} \mathbb{S}_{\mu'}A\ot A\xrightarrow{}
\mathbb{S}_{\mu'}A\ot V\xrightarrow{}\mathbb{S}_{\mu'}A\ot
B\xrightarrow{}0.
\end{equation}
Tensoring $E_1$ by $\mathbb{S}_{\nu}B$ on the right yields the
extension
\begin{equation}\label{exseq2}
E_2:\,\,\,\, 0\xrightarrow{} \mathbb{S}_{\mu'}A\ot A\ot
\mathbb{S}_{\nu}B\xrightarrow{} \mathbb{S}_{\mu'}A\ot V\ot
\mathbb{S}_{\nu}B\xrightarrow{}\mathbb{S}_{\mu'}A\ot B\ot
\mathbb{S}_{\nu}B\xrightarrow{}0.
\end{equation}

Set
\begin{equation}\label{munu}
\mu'_+:= \{\mu\in\mu'+1\mid N_{\mu,\nu}^{\lambda}\ne 0\},\,\,\,
\nu_+:= \{\nu'\in\nu+1\mid N_{\mu',\nu'}^{\lambda}\ne 0\}.
\end{equation}

The following lemma is clear.

\begin{lemma}\label{l}
Let $1\le i\le n$ be an integer, and let $\mu'\in \lambda-(n-i+1)$,
$\nu\in \lambda-i$. Then for any $\mu\in \mu'_+$ and $\nu'\in
\nu_+$, the triple $(c_{\mu}\ot id_{\mathbb{S}_{\nu}B},c_{\mu}\ot
id_{\mathbb{S}_{\nu}B},id_{\mathbb{S}_{\mu'}A}\ot c_{\nu'})$ defines
a morphism of extensions $E_2\to (V_{i}^2)$:
\begin{equation*}
\begin{CD} 0 @>>> \mathbb{S}_{\mu'}A\ot A\ot
\mathbb{S}_{\nu}B @>>> \mathbb{S}_{\mu'}A\ot V\ot \mathbb{S}_{\nu}B
@>>> \mathbb{S}_{\mu'}A\ot B\ot \mathbb{S}_{\nu}B @>>> 0\\
@. @VVc_{\mu}\ot id_{\mathbb{S}_{\nu}B}V @VVc_{\mu}\ot
id_{\mathbb{S}_{\nu}B}V @VVid_{\mathbb{S}_{\mu'}A} \ot c_{\nu'}V @.\\
0 @>>> V_i @>>> V_{i}^2 @>>> V_{i-1}@>>> 0.
\end{CD}
\end{equation*}
\qed
\end{lemma}

Fix an integer $1\le i\le n$, and $\mu'\in \lambda-(n-i+1)$, $\nu\in
\lambda-i$. For any $\mu\in \mu'_+$ and $\nu'\in \nu_+$, define the
following two subsets of the ground field $k$:
\begin{equation}\label{ai}
A_i(\mu',\mu,\nu):=\{d\mid p_{\mu',\mu}(d)=0\}\subseteq k
\end{equation}
and
\begin{equation}\label{bi}
B_i(\mu',\nu,\nu'):=\{d\mid p_{\nu,\nu'}(d)=0\}\subseteq k.
\end{equation}

\begin{example}\label{extreme}
By convention, $\lambda-n=\{(0)\}$. Therefore, for any $\mu',\nu\in
\lambda-1$, we have that
$A_1((0),(1),\nu)=B_n(\mu',(0),(1))=\emptyset$. On the other
extreme, by Corollary \ref{cor1},
$B_1((0),\nu,\lambda)=A_n(\mu',\lambda,(0))=\{1-q(\lambda),\dots,p(\lambda)-1\}$
if $\lambda$ is not a hook, and
$B_1((0),\nu,\lambda)=A_n(\mu',\lambda,(0))
=\{1-q(\lambda),\dots,-1,1,\cdots,p(\lambda)-1\}$ if $\lambda$ is a
hook.
\end{example}

Set $a:=\dim A$, $b:=\dim B$ for the rest of the paper.

\begin{lemma}\label{ll}
Let $1\le i\le n$ be an integer, and let $\mu'\in \lambda-(n-i+1)$,
$\nu\in \lambda-i$. Let $\mu\in \mu'_+$, and let $\nu'\in \nu_+$ be
such that $b\notin B_i(\mu',\nu,\nu')$. Then $(c_{\mu}){_*}(E_1)=0$
in $\Ext(\mathbb{S}_{\mu'}A\ot B,\mathbb{S}_{\mu}A)$.
\end{lemma}

\begin{proof}
By Lemma \ref{l} and a standard fact on extensions (see e.g.,
\cite{m}), we have that $(c_{\mu}\ot id)_*(E_2)=(id\ot
c_{\nu'})^*(V_{i}^2)$. Since by Lemma \ref{splits}, $(V_{i}^2)=0$,
we have that $(c_{\mu}\ot id)_*(E_2)=0$ in
$\Ext(\mathbb{S}_{\mu'}A\ot B\ot
\mathbb{S}_{\nu}B,\mathbb{S}_{\mu}A\ot \mathbb{S}_{\nu}B)$.

Let
$$f:\Ext(\mathbb{S}_{\mu'}A\ot B,\mathbb{S}_{\mu}A)\to
\Ext(\mathbb{S}_{\mu'}A,\mathbb{S}_{\mu}A\ot B^*)$$ be the
isomorphism given by Lemma \ref{ext}, let
\begin{equation*}
\Ext(\mathbb{S}_{\mu'}A,\mathbb{S}_{\mu}A\ot B^*)\xrightarrow{(id\ot
coev_{\mathbb{S}_{\nu}B^*})_*}
\Ext(\mathbb{S}_{\mu'}A,\mathbb{S}_{\mu}A\ot B^*\ot
\mathbb{S}_{\nu}B^*\ot \mathbb{S}_{\nu}B),
\end{equation*}
and let
$$g:\Ext(\mathbb{S}_{\mu'}A,\mathbb{S}_{\mu}A\ot B^*\ot
\mathbb{S}_{\nu}B^*\ot \mathbb{S}_{\nu}B)\to
\Ext(\mathbb{S}_{\mu'}A\ot B\ot
\mathbb{S}_{\nu}B,\mathbb{S}_{\mu}A\ot \mathbb{S}_{\nu}B)$$ be the
isomorphism given by Lemma \ref{ext} (composed with the appropriate
commutativity constraints). Then, it is straightforward to verify
that
$$0=(c_{\mu}\ot id)_*(E_2)=\left(g\circ (id\ot
coev_{\mathbb{S}_{\nu}B^*})_*\circ f\right)((c_{\mu}){_*}(E_1)).$$

Now, by our assumption on $b$ and Theorem \ref{c1}, the morphism
$\theta_{B^*,n-i+1,\nu,\nu'}$ is a split injection. Therefore, by
Lemma \ref{cor}, $(id\ot coev_{\mathbb{S}_{\nu}B^*})_*$ is
injective, and the result follows.
\end{proof}

We are now ready to prove the key proposition for the proof of
Theorem \ref{main}.
\begin{proposition}\label{mainp}
Assume there exist an integer $1\le i\le n$, a pair of partitions
$\mu'\in \lambda-(n-i+1)$, $\nu\in \lambda-i$ and a pair of
partitions $\mu\in \mu'_+$, $\nu'\in \nu_+$, such that $a\notin
A_i(\mu',\mu,\nu)$ and $b\notin B_i(\mu',\nu,\nu')$. Then $(V)=0$ in
$\Ext(B,A)$.
\end{proposition}

\begin{proof}
By Theorem \ref{c1}, the morphisms $\theta_{A,i-1,\mu',\mu}$ and
$\theta_{B,n-i+1,\nu,\nu'}$ are split injections. Consider now the
following commutative diagram:
\begin{equation*}
\begin{CD} \Ext(B,A) @> f>> \Ext(B,A\ot \mathbb{S}_{\mu'}A\ot
\mathbb{S}_{\mu'}A^*) @> g>> \Ext(B,
\mathbb{S}_{\mu}A\ot \mathbb{S}_{\mu'}A^*) \\
@. @VV\cong V @VV\cong V \\
@. \Ext(B\ot \mathbb{S}_{\mu'}A,A\ot \mathbb{S}_{\mu'}A)
@>(c_{\mu}){_*}>> \Ext(B\ot \mathbb{S}_{\mu'}A, \mathbb{S}_{\mu}A),
\end{CD}
\end{equation*}
where $f:=(id_A\ot coev_{\mathbb{S}_{\mu'}A})_*$, $g:=(c_{\mu}\ot
id_{\mathbb{S}_{\mu'}A^*})_*$, and the two vertical isomorphisms are
given by Lemma \ref{ext}. Observe that
$gf=(\theta_{A,i-1,\mu',\mu})_*$ is injective. It is now clear that
the proposition follows from Lemma \ref{cor} and Lemma \ref{ll}.
\end{proof}

\section{The proof of Theorem \ref{main}}

\subsection{$F(\lambda)\subseteq G(\lambda)$:}
We have to show that if $d\in F(\lambda)$ then $d\in G(\lambda)$,
i.e., that there exists a symmetric rigid tensor category $\C$ over
$k$ with a non-semisimple object $V$ of dimension $d$ such that
$\mathbb{S}_{\lambda}V$ is semisimple. This follows from the
following two observations.

1) Let $r,\,s$ be nonnegative integers such that $r+s\ge 2$. One can
introduce on the superspace $V:=\mathbb{C}^{r\mid s}$ a structure of
a nonsemisimple representation of some supergroup (e.g., the
supergroup of upper triangular matrices). On the other hand, if
$\lambda$ contains a box $(r+1,s+1)$ then $\mathbb{S}_{\lambda}V=0$
(see e.g., \cite{d2}), so $\mathbb{S}_{\lambda}V$ is automatically
semisimple while $V$ is not.

2) Suppose $\lambda=(q^p)$ is a rectangle. If $V$ is a nonsemisimple
group representation of dimension $p>1$, then $\mathbb{S}_\lambda
V=(\wedge^{p}V) ^{\otimes q}$ is $1-$dimensional, so is
automatically semisimple, while $V$ is not. Finally, for $-q$, use
now Remark \ref{trick}.

We are done.

\subsection{$G(\lambda)\subseteq F(\lambda)$:} Let $\C$ be any symmetric rigid
tensor category over $k$, and let $V\in\C$ be an object of $\C$. We
have to show that if $\dim V\notin F(\lambda)$ and
$\mathbb{S}_{\lambda}V$ is semisimple then so is $V$ (i.e., $\dim
V\notin G(\lambda)$). To this end, it is enough to show that if
$\dim V\notin F(\lambda)$ then there exist an integer $1\le i\le n$,
a pair of partitions $\mu'\in \lambda-(n-i+1)$, $\nu\in \lambda-i$,
and a pair of partitions $\mu\in \mu'_+$, $\nu'\in \nu_+$,
satisfying the conditions of Proposition \ref{mainp}.

Let $\lambda^*$ denote the conjugate of $\lambda$. Write
$\lambda=(\lambda_1,\dots,\lambda_p)$ and
$\lambda^*=(\lambda'_1,\dots,\lambda'_q)$, where $q=\lambda_1
\ge\cdots\ge\lambda_p\ge 1$ and $p=\lambda'_1
\ge\cdots\ge\lambda'_q\ge 1$.

\subsubsection{\textbf{The general case.}} In this subsection we prove that if
$\dim V\notin R(\lambda)$ then the exact sequence $(V)$ splits.

If $i=1$, $A_1((0),(1),\nu)=\emptyset$ for any $\nu\in \lambda-1$
(see Example \ref{extreme}), so there is no condition on $a$.
Therefore, if $b$ {\em is not} equal in $k$ to an element of
$B_1((0),(1),\nu)$ for some $\nu\in \lambda-1$, we are done. So
suppose $b$ is equal in $k$ to some element of
$$B_1((0),(1),\nu)=\{1-q,\dots,p-1\},$$ which we shall continue to
denote by $b$ (so now $b\in \mathbb{Z}$).

\textbf{\emph{Subcase 1.}} Suppose that $b>0$, and set $i:=p-b+1$;
then $2\le i \le p$. Let $\mu':=(\lambda_{p-i+1},\dots,\lambda_p-1)$
be the last $i$ rows of $\lambda$ without the last box, and let
$\nu:=(\lambda_{1},\dots,\lambda_{p-i})$ be the first $p-i$ rows of
$\lambda$. Let $\mu:=(\lambda_{p-i+1},\dots,\lambda_p)$ and let
$\nu':=(\lambda_{1},\dots,\lambda_{p-i},1)$. It follows easily from
the Littlewood-Richardson rule that $\mu\in\mu'_+$ and
$\nu'\in\nu_+$.

We now use (\ref{new}) to find out that
\begin{equation}\label{a1}
A_{n_i}(\mu',\mu,\nu)= \{d\in k\mid p_{\mu',\mu}(d)=0\}=
\{1-\lambda_{p-i+1},\dots,i-1\}
\end{equation}
and
\begin{equation}\label{b1}
B_{n_i}(\mu',\nu,\nu')= \{d\in k\mid p_{\nu,\nu'}(d)=0\}=
\{1-q,\dots,p-i\},
\end{equation}
where $n_i:=\sum_{j=p-i+1}^p \lambda_j$. We therefore see that
$b=p-i+1\notin B_{n_i}(\mu',\nu,\nu')$. Now, if $a\notin
A_{n_i}(\mu',\mu,\nu)$, we are done. Otherwise, we are done by our
assumption on $\dim V$ (since $\dim V=a+b$).

\textbf{\emph{Subcase 2.}} Suppose that $b=0$. Then $0\notin
B_n(\mu',(0),(1))=\emptyset$ for any $\mu'\in \lambda-1$. Now, if
$a\notin A_n(\mu',(0),(1))$, we are done. Otherwise, we are done by
our assumption on $\dim V$.

\textbf{\emph{Subcase 3.}} Suppose that $b<0$, and set $i:=q+b+1$;
then $2\le i \le q$. Let
$\mu':=(\lambda'_{q-i+1},\dots,\lambda'_q-1)$ be the last $i$
columns of $\lambda$ without the last box,
and let $\nu:=(\lambda'_{1},\dots,\lambda'_{q-i})$ be the first
$q-i$ columns of $\lambda$. Let
$\mu:=(\lambda'_{q-i+1},\dots,\lambda'_q)$ and let
$\nu':=(\lambda'_{1},\dots,\lambda'_{q-i},1)$. It follows easily
from the Littlewood-Richardson rule that $\mu\in\mu'_+$ and
$\nu'\in\nu_+$.

We now use (\ref{new}) to find out that
\begin{equation}\label{a2}
A_{n_i}(\mu',\mu,\nu)= \{d\in k\mid p_{\mu',\mu}(d)=0\}=
\{1-i,\dots,\lambda'_{q-i+1}-1\}
\end{equation}
and
\begin{equation}\label{b2}
B_{n_i}(\mu',\nu,\nu')= \{d\in k\mid p_{\nu,\nu'}(d)=0\}=
\{i-q,\dots,p-1\},
\end{equation}
where $n_i:=\sum_{j=q-i+1}^q \lambda'_j$. We therefore see that
$b=i-q-1\notin B_{n_i}(\mu',\nu,\nu')$. Now, if $a\notin
A_{n_i}(\mu',\mu,\nu)$, we are done. Otherwise, we are done by our
assumption on $\dim V$.

\subsubsection{\textbf{The non-rectangle case.}} Assume $\lambda$ is not a
rectangle. We have to show that $\dim V=p$ and $\dim V=-q$ are
allowed.
Let $b,\,i$ be as in Subcase 1 of 5.2.1.

Let $\mu':=(\lambda_{p-i+2},\dots,\lambda_p)$ and
$\nu:=(\lambda_{1},\dots,\lambda_{p-i+1}-1)$ be the last $i-1$ rows
of $\lambda$ and the first $p-i+1$ rows of $\lambda$ without the
last box, respectively. Choose $\mu\in \mu'_+$ with $i-1$ rows (it
exists since $\lambda$ is not a rectangle!) and let
$\nu':=(\lambda_{1},\dots,\lambda_{p-i+1})$. It follows easily from
the Littlewood-Richardson rule that $\nu'\in\nu_+$. Moreover, we now
have that $A_{n_i}(\mu',\mu,\nu)= \{1-\lambda_{p-i+2},\dots,i-2\}$,
where $n_i:=1+\sum_{j=p-i+2}^p \lambda_j$. Hence, $i-1\notin
A_{n_i}(\mu',\mu,\nu)$ and $b=p-i+1\notin B_{n_i}(\mu',\nu,\nu')$.
We thus conclude that $\dim V=p$ is allowed in this case, as
claimed.

The claim that $\dim V=-q$ is allowed follows now from Remark
\ref{trick}.

\subsubsection{\textbf{The case $(3,2)\notin \lambda$ or $(2,3)\notin
\lambda$.}} Suppose $(3,2)\notin \lambda$. We have to show that
$\dim V=1$ is allowed.

\textbf{\emph{Subcase 1.}} Let $b,\,i$ be as in Subcase 1 of 5.2.1.

First note that for $2\le i\le p-2$, $\lambda_{p-i+1}=1$. Hence
$b\notin B_{n_i}(\mu',\mu,\nu)$ and $1-b\notin
A_{n_i}(\mu',\mu,\nu)$ (see (\ref{a1}), (\ref{b1})).

Now, for $i=p-1$ (so $b=2$), take $\mu':=(1^{p-1})$, $\mu:=(1^{p})$,
$\nu:=(q-1,\lambda_2-1)$ and $\nu':=(q,\lambda_2-1)$. It follows
easily from the Littlewood-Richardson rule that $\mu\in\mu'_+$ and
$\nu'\in\nu_+$. Moreover, $-1\notin A_p(\mu',\mu,\nu)$ and $2\notin
B_p(\mu',\nu,\nu')$.

For $i=p$ (so $b=1$), take $\mu':=(q,1^{p-2})$, $\mu:=(q,1^{p-1})$,
$\nu:=(\lambda_2-1)$ and $\nu':=(\lambda_2)$. It follows easily from
the Littlewood-Richardson rule that $\mu\in\mu'_+$ and
$\nu'\in\nu_+$. Moreover, $0\notin A_{p+q-1}(\mu',\mu,\nu)$ and
$1\notin B_{p+q-1}(\mu',\nu,\nu')$.

\textbf{\emph{Subcase 2.}} Let $b,\,i$ be as in Subcase 2 of 5.2.1.

Take $\mu':=(\lambda_2-1)$, $\mu:=(\lambda_2)$, $\nu:=(q,1^{p-2})$
and $\nu':=(q,1^{p-1})$. It follows easily from the
Littlewood-Richardson rule that $\mu\in\mu'_+$ and $\nu'\in\nu_+$.
Moreover, $1\notin A_{\lambda_2}(\mu',\mu,\nu)$ and $0\notin
B_{\lambda_2}(\mu',\nu,\nu')$.

\textbf{\emph{Subcase 3.}} Let $b,\,i$ be as in Subcase 3 of
Subsection 5.2.1.

First note that for $2\le i\le q-2$, $\lambda'_{q-i+1}=1$. Hence
$b\notin B_{n_i}(\mu',\mu,\nu)$ and $1-b\notin
A_{n_i}(\mu',\mu,\nu)$ (see (\ref{a2}), (\ref{b2})).

Now, for $i=q-1,\,q$ (so $b=-2,\,-1$), take
$\mu':=(\lambda_1,\lambda_2-1)$, $\mu:=(\lambda_1,\lambda_2)$,
$\nu:=(1^{p-2})$ and $\nu':=(1^{p-1})$. It follows easily from the
Littlewood-Richardson rule that $\mu\in\mu'_+$ and $\nu'\in\nu_+$.
Moreover, $2,3\notin A_{n+2-p}(\mu',\mu,\nu)$ and $-1,-2\notin
B_{n+2-p}(\mu',\nu,\nu')$.

We therefore conclude that $\dim V=1$ is allowed in this case, as
claimed.

Finally, the claim that $\dim V=-1$ is allowed in the case
$(2,3)\notin \lambda$ follows now from Remark \ref{trick}.

\subsubsection{\textbf{The hook case.}} Assume $\lambda$ is a hook. We have to
show that $\dim V=0$ is allowed.

\textbf{\emph{Subcase 1.}} Let $b,\,i$ be as in Subcase 1 of 5.2.1.

Since $\lambda_{p-i+1}=1$ for $i<p$, we get from (\ref{a1}) that
$-b\notin A_{n_i}(\mu',\mu,\nu)$. On the other hand, for $i=p$ (so
$b=1$), take $\mu':=(1^{p-1})$, $\mu:=(1^{p})$, $\nu:=(q-1)$ and
$\nu':=(q)$. It follows easily from the Littlewood-Richardson rule
that $\mu\in\mu'_+$ and $\nu'\in\nu_+$. Moreover, $-1\notin
A_p(\mu',\mu,\nu)$ and $1\notin B_p(\mu',\nu,\nu')$.

\textbf{\emph{Subcase 2.}} Let $b,\,i$ be as in Subcase 3 of 5.2.1.

Since $\lambda'_{q-i+1}=1$ for $i<q$, we get from (\ref{a2}) that
$-b\notin A_{n_i}(\mu',\mu,\nu)$. On the other hand, for $i=q$ (so
$b=-1$), take $\mu':=(q-1)$, $\mu:=(q)$, $\nu:=(1^{p-1})$ and
$\nu':=(1^{p})$. It follows easily from the Littlewood-Richardson
rule that $\mu\in\mu'_+$ and $\nu'\in\nu_+$. Moreover, $1\notin
A_q(\mu',\mu,\nu)$ and $-1\notin B_q(\mu',\nu,\nu')$.

We therefore conclude that $\dim V=0$ is allowed in this case, as
claimed.

This concludes the proof of the theorem. \qed

\end{document}